\documentclass[12pt,headings=big]{scrartcl}
 \usepackage{cmap}
 \usepackage{amsmath,amsxtra,amscd,amssymb,latexsym,stmaryrd,amsthm,lipsum}
 \usepackage{mathrsfs, calc, xcolor}

\usepackage[colorlinks=true,linkcolor=red!70!black,citecolor=green!70!black,
    urlcolor=magenta!70!black,backref]{hyperref}

  \usepackage{lmodern}
  \usepackage[utf8]{inputenc}
  \usepackage[T1]{fontenc}

\addtokomafont{title}{\rmfamily\itshape}
\theoremstyle{plain}
\newtheorem{theo}{Theorem}[section]
\newtheorem*{theo*}{Theorem}
\newtheorem{coro}[theo]{Corollary}
\newtheorem{prop}[theo]{Proposition}
\newtheorem{lemm}[theo]{Lemma}
\newtheorem{theomain}[theo]{Theorem}
\newtheorem{coromain}[theo]{Corollary}

\theoremstyle{definition}
\newtheorem{defi}[theo]{Definition}
\newtheorem{rema}[theo]{Remark}

\newcommand*{\dd}%
  {\relax\ifnum\lastnodetype>0\mskip\medmuskip\fi\mathrm{d}}
\newcommand{\fspace}[1]{\mathcal{#1}}
\newcommand{\fX}{\fspace{X}}
\newcommand{\fY}{\fspace{Y}}
\newcommand{\BX}{\operatorname{\fspace{B}}(\fX)}

\newcommand{\op}[1]{\mathrm{#1}}

\newcommand{\Id}{\operatorname{Id}}

\newcommand{\llip}{\operatorname{\lvert D\rvert}}

\newlength{\hypobox}
\newlength{\gapbox}
\newcounter{hypo}
\renewcommand{\thehypo}{\textup{(H\arabic{hypo})}}

\newcounter{hypop}
\renewcommand{\thehypop}{\textup{(H\arabic{hypop}')}}

\newcounter{gap}
\renewcommand{\thegap}{\textup{(SG\arabic{gap})}}

\title{Effective perturbation theory \\ for simple isolated eigenvalues \\ of
linear operators}

\author{Beno\^{\i}t R. Kloeckner \thanks{Universit\'e Paris-Est, Laboratoire d'Analyse et de Mat\'ematiques Appliqu\'ees (UMR 8050), UPEM, UPEC, CNRS, F-94010, Cr\'eteil, France}}

\begin{document}
%%%%%%%%%%%%%%%%%%%%%%%%%%%%%%%%%%%%%%%%%%%%%%%%%%%%%%%%%%%%%%%
%%%%%%%%%%%%%%%%%%%%%%%%%%%%%%%%%%%%%%%%%%%%%%%%%%%%%%%%%%%%%%%
%%%%%%%%%%%%%%%%%%%%%%%%%%%%%%%%%%%%%%%%%%%%%%%%%%%%%%%%%%%%%%%

\maketitle

\begin{abstract}
We propose a new approach to the  spectral theory of perturbed linear operators in the case of a simple isolated eigenvalue. We obtain two kinds of results: ``radius bounds'' which ensure perturbation theory applies for perturbations up to an explicit size, and ``regularity bounds'' which control the variations of eigendata to any order. Our method is based on the Implicit Function Theorem and proceeds by establishing differential inequalities on two natural quantities: the norm of the projection to the eigendirection, and the norm of the reduced resolvent. We obtain completely explicit results without any assumption on the underlying Banach space.

In companion articles, on the one hand we apply the regularity bounds to Markov chains, obtaining  non-asymptotic concentration and Berry-Esseen inequalities with explicit constants, and on the other hand we apply the radius bounds to transfer operators of intermittent maps, obtaining explicit high-temperature regimes where a spectral gap occurs.
\end{abstract}

%\setcounter{tocdepth}{1}
%\tableofcontents

%%%%%%%%%%%%%%%%%%%%%%%%%%%%%%%%%%%%%%%%%%%%%%%%%%%%%%%%%%%%%%%%%
\section{Introduction}\label{sec:intro}

Let $\fX$ be a real or complex Banach space and denote by $\mathbb{K}$ the field of scalars and by $\BX$ the space of bounded linear operator acting on $\fX$, endowed with the operator norm.
Given an operator $\op{L}_0\in\BX$,\footnote{The case of a closed operator can be treated similarly using the graph norm on its domain.} it is a natural and old set of problems to ask how its spectral properties change under perturbation, i.e. when one considers $\op{L}=\op{L}_0+\op{M}$ where $\op{M}$ is small in operator norm. A particularly important question with many applications, for example in the study of Markov chains and of transfer operators of dynamical systems, is the analytic dependency of a simple, isolated eigenvalue with the perturbation.

This question is often considered for ``Gâteaux'' perturbations, i.e. of the form $t\mapsto \op{L}_t$, $t\in(-\varepsilon,\varepsilon)$ and in a purely asymptotic form. At least in some fields, authors often refer to the book of Kato \cite{Kato} (see also \cite{DS}), without using the quantitative statements that are present there (written in the finite-dimension chapters, but adaptable to infinite-dimensional spaces). Kato uses contour integrals, as introduced by Sz.-Nagy \cite{Sz.-Nagy} but it has been noticed by Rosenbloom \cite{Rosenbloom} that one can use the implicit function theorem to easily obtain similar results, and this will serve as a starting point.

We shall distinguish two types of quantitative statements: one can estimate the allowable size of a perturbation below which an analytic simple isolated eigenvalue is ensured (\emph{radius estimates}), or bound the variations or the iterated derivatives of the eigenvalue and other eigendata (\emph{regularity estimates}). Radius estimate are present in various works. We refer to \cite{Baumgartel}, notably page 322, for an account which is about as complete as we could give; more recent references are \cite{Farid} and \cite{Nair}.
Regularity estimates are much less common, the only one I know of being in \cite{Rosenbloom} (Corollary 1a.); it is quite involved and insufficient for some applications (see e.g. Remark \ref{rema:BE}). 

The goal of this article is to produce radius and regularity estimates that:
\begin{enumerate}
\item hold in any Banach space,
\item are uniform over directions (as in ``Fréchet'' dérivatives), non-asymptotic and entirely explicit,
\item control higher order derivatives of eigendata (for the regularity estimates).
\end{enumerate}
The motivations for these properties lie in the applications developed in dynamical systems \cite{K:HT} and probability theory \cite{K:concentration}, which both rely on the present article. The operators to be perturbed (transfer operator and Markov averaging operator) are not self-adjoint and in fact very often act on spaces which are not Hilbertian, thus preventing the use of a large part of the literature (e.g. pseudo-spectrum). The uniformity and explicitness is also crucial, as the applications are about effective results. The control of higher order derivatives is crucial to \cite{K:concentration}, where the spectral method is needed up to order $3$. To the best of my knowledge none of the numerous previously known estimates have these qualities.

To achieve this we mostly rely on the comparison principle for differential inequalities, an approach that feels simpler than the \emph{majorizing series} method.
We also want to argue for dropping the parametrized approach (which considers a map $t\mapsto \op{L}_t)$ to perturbation theory in favor of a direct approach inside $\BX$ (specific parametrized perturbation being then handled by composition), as it seems to clarify the computations.
As a testimony to this point of view, let us give right away a short proof of the qualitative perturbation theory of a simple isolated eigenvalue. The proof does not differ substantially from the one given in \cite{Rosenbloom}, but we include it with our notation as it serves as a starting point to our new results, and also to advertise further the point made by Rosenbloom that this approach should simplify the matter at hand. Note that a similar approach is taken in \cite{HH}.

\begin{theo*}\label{theo:perturbation+}
If $\op{L}_0\in\BX$ has a simple isolated eigenvalue, then there is an open neighborhood $\fspace{V}$ of $\op{L}_0$ such that
all $\op{L}\in\fspace{V}$ have an eigenvalue $\lambda_\op{L}$ close to $\lambda_0$. 
The map $\lambda:\fspace{V}\to\mathbb{K}$ is analytic, $\op{L}$ does not have other eigenvalues near $\lambda_0$, and there is another analytic map $u:\fspace{V}\to \fX$ such that $u_\op{L}$ is an eigenvector of $\op{L}$ for $\lambda_\op{L}$.
\end{theo*}

\begin{proof}[Proof (Rosenbloom)]
Denote by $u_0$ an eigenvector, let $\phi_0$ be an eigenform (i.e. $\phi_0$ is an eigenvector of $\op{L}_0^*$ for $\lambda_0$) and up to multiplying either of them by a scalar assume $\phi_0(u_0)=1$.
Consider the obviously analytic map
\begin{align*}
F : \BX \times (\fX \times \mathbb{K}) &\to \fX \times \mathbb{K} \\
  (\op{L},u,\lambda) &\mapsto (\op{L}u-\lambda u,\, \phi_0(u)-1)
\end{align*}
We have $F(\op{L}_0,u_0,\lambda_0)=0$ and the partial derivative of $F$ with respect to the $(\fX \times \mathbb{K})$ factor at the point $(\op{L}_0,u_0,\lambda_0)$ is
\[\partial_2 F_{0}(v,\rho) = ((\op{L}_0-\lambda_0)v-\rho u_0, \phi_0(v)).\]
Decomposing along $\langle u_0\rangle\oplus \ker \phi_0$ we see that
\[\partial_2 F_{0}(au_0+k,\rho) = ((\op{L}_0-\lambda_0)k-\rho u_0, a)\]
so that for all $b,\eta\in\mathbb{K}$ and $h\in \ker\phi_0$ the equation $\partial_2 F_{0}(au_0+k,\rho) = (bu_0+h,\eta)$ has a unique solution 
\[ a=\eta, \quad \rho=-b,\quad k= (\op{L}_0-\lambda_0)_{|\ker \phi_0}^{-1}h \]
where the invertibility of $(\op{L}_0-\lambda_0)$ from $\ker\phi_0$ to itself follows from the fact that $\lambda_0$ is simple isolated. The Implicit Functions Theorem then ensures that there is an analytic map $(u,\lambda):\fspace{V} \to \fX\times \mathbb{K}$ defined in a neighborhood of $\op{L}_0$ such that $F(\op{L},u_\op{L},\lambda_\op{L})\equiv 0$ and $(u_\op{L},\lambda_\op{L})$ is the unique solution to this equation in a neighorhood of $(u_0,\lambda_0)$. In particular $\lambda_\op{L}$ is an eigenvalue of $\op{L}$ and $u_\op{L}$ is an eigenvector.
\end{proof}

%%%%%%
\paragraph*{Organization of the article.}

In Section \ref{sec:statements} we fix some notation and gather our main statements, to ease later reference.
Section \ref{sec:prer} gives information on a few tools we need: analyticity in Banach spaces, the Implicit Function Theorem, and metric derivatives. In Section \ref{sec:derivatives} we give formulas for the derivatives of the eigendata, from which in Section \ref{sec:estimates} we derive Lipschitz estimates on eigendata and on two crucial parameters $\tau$, $\gamma$. Last,
Section \ref{sec:proofs} contains the end of the proofs of the main results.

%%%%%%%%%%%%%%%%%%%%%%%%%%%%%%%%%%%%%%%%%%%%%%%%%%%%%%%%%%%%%%%%%
\section{Main results}\label{sec:statements}

%%%%%%
\paragraph*{Notation and convention.}

All norms will be denoted by $\lVert\cdot\rVert$. Operators, linear and multilinear forms will always be endowed with the operator norm. We denote by $B(y,r)$ the ball of radius $r$ and center $y$.

We shall denote the composition of operators and the application of an operator to a vector by simple juxtaposition (as in $\pi_0\op{M}u_0$) unless it feels necessary to mark them with $\circ$ and parentheses for more clarity (e.g. $\pi_0\circ \op{M}(u_0)$). When $\psi\in\fX^*$ and $v\in\fX$, we will write $\psi(\cdot)v$  for the rank-one operator $\fX\to\fX$ mapping $x$ to $\psi(x)v$ (often denoted by $v\otimes\psi$), not to be confused with the scalar $\psi v$.

In the complex case, we take the convention that $\fX^*$ is made of linear forms (not semi-linear) and we pair forms with vectors without taking conjugate, i.e. $\langle \psi, v\rangle = \psi v$, so that adjoints have the same spectrum as the original operator instead of the conjugate one.

To state more conveniently some of our results, we will use the following two variations on the big-$O$ notation. First, $O_C(\cdot)$ will mean a big-$O$ with explicit bound: if $f$ is a Banach-valued map and $g$ is a function,
\[f = O_C(g) \quad\mbox{if and only if}\quad  \lVert f(x)\rVert \le C \lvert g(x)\rvert \quad \forall x.\]
Second, we write $f = O_{*a,b,\dots}(g)$ when for all 
$a_+,b_+,\dots$, there exist a constant $C=C(a_+,b_+,\dots)>0$ such that for all arguments $x$ with $a(x)\le a_+, b(x)\le b_+,\dots$ we have
\[\lVert f(x)\rVert \le C \lvert g(x) \rvert.\]

All maps with operator-valued arguments have their operators written in subscript indices, e.g. we write $u_\op{L}$, $\lambda_\op{L}$ rather than $u(\op{L})$, $\lambda(\op{L})$; the index $0$ refers to $\op{L}_0$ in this notation, e.g. $u_0=u_{\op{L}_0}$.

Among the possible equivalent definitions of a simple isolated eigenvalue the following one is closest to our needs.
\begin{defi}
We say that $\op{L}\in\BX$ has the scalar $\lambda$ as a \emph{simple isolated eigenvalue} if there exists a non-zero $u\in\fX$ such that $\op{L}u=\lambda u$ and if there exists a complement $G$ to $\langle u\rangle$ which is preserved by $\op{L}$ and such that the restriction and co-restriction of $\op{L}-\lambda$ to $G\to G$ is invertible.
\end{defi}
Note that by a complement we shall always mean a topological (i.e. closed) complement and that we write $\lambda$ for the scalar operator $\lambda\Id$ when no confusion is possible. In the above circumstances, we will denote by $(\op{L}-\lambda)^{-1}$ the inverse of $\op{L}-\lambda$ viewed as an operator on $G$. From now on, we will write all eigendata for $\op{L}$ with a subscript, implicitly assuming $\op{L}$ is in a sufficiently small neighborhood of $\op{L}_0$. In particular, the subspace $G$ will be denoted by $G_\op{L}$ and will be called the \emph{stable complement} of $\op{L}$.

If $\lambda_\op{L}$ is a simple isolated eigenvalue of $\op{L}$, it is also a simple isolated eigenvalue of the dual operator $\op{L}^*$, which has an eigenform $\phi_\op{L}\in\fX^*$ (i.e. $\phi_\op{L}\op{L}(x)=\lambda_\op{L} x$ for all $x\in\fX$). The stable complement $G_\op{L}$ coincides with the kernel of $\phi_\op{L}$, and the $\op{L}^*$-stable complement of $\phi_\op{L}$ is the the set $u_\op{L}^\perp$ of forms that vanish on the eigenvector $u_\op{L}$.

If we normalize the eigenvector or eigenform such that $\phi_\op{L}(u_\op{L})=1$, we can write $\op{P}_\op{L}=\phi_\op{L}(\cdot)u_\op{L}$ and $\pi_\op{L}=\Id-\phi_\op{L}(\cdot)u_\op{L}$ for the projections with respect to the decomposition $\fX=\langle u_\op{L}\rangle \oplus G_\op{L}$.

The \emph{reduced resolvent} (at $\lambda_\op{L}$) of $\op{L}$ is the operator
\[\op{S}_\op{L} = (\op{L}-\lambda_\op{L})^{-1}\pi_\op{L} \in\BX,\]
which takes its values in $G_\op{L}$.

Our method relies on two particular quantities associated to an operator (with a simple isolated eigenvalue), which on the one hand control all derivatives of eigendata (to all order), and on the other hand are defined in terms of some eigendata. 

\begin{defi}\label{defi:quantities}
Let $\op{L}\in\BX$ have a simple isolated eigenvalue $\lambda_\op{L}$, with eigenvector $u_\op{L}$ and eigenform $\phi_\op{L}$. 
We shall consider the quantities
\[\tau_\op{L}:=\frac{\lVert \phi_\op{L}\rVert \lVert u_\op{L}\rVert}{\lvert \phi_\op{L} u_\op{L}\rvert} = \lVert\op{P}_\op{L}\rVert \quad\mbox{and}\quad \gamma_\op{L}:=\lVert(\op{L}-\lambda_\op{L})^{-1}\pi_\op{L}\rVert=\lVert S_\op{L}\rVert,\]
respectively called the \emph{condition number} and the \emph{spectral isolation}.
\end{defi}
Let us quickly explain their relevance. A large condition number means that the eigenspace $\langle u\rangle$ is close to the stable complement $G_\op{L}$; when $\fX$ is a Hilbert space and $\op{L}$ is normal the condition number is $1$, but the condition number is also $1$ in many other cases. The spectral isolation controls how far $\lambda$ must be from the rest of the spectrum (small $\gamma$ entails a very isolated eigenvalue). 

It will be convenient to say that every $\op{L}$ in an open connected set $\fspace{V}\ni\op{L}_0$ ``has an ASIE'' (standing for Analytic Simple Isolated Eigenvalue) if there is an analytic map $\lambda$ defined on $\fspace{V}$ such that $\lambda_\op{L}$ is a simple isolated eigenvalue of $\op{L}$. It will then follow that the other eigendata will also be analytic in the same region.

%%%%%%%%%%%%%%%%%%%%%%%%%%%%%%%%%%%%%%%%%%%%%%%%%%%%%%%%%%%%%%%%%
\subsection{Radius estimate}

We assume $\op{L}_0\in\BX$ has a simple isolated eigenvalue $\lambda_0$ with eigenvector $u_0$, eigenform $\phi_0$, stable complement $G_0:=\ker\phi_0$ and associated projections $P_0$, $\pi_0$.

Our first result is a simple radius estimate.
\begin{theomain}\label{theo:perturbation}
All $\op{L}$ such that $\displaystyle \lVert\op{L}-\op{L}_0\rVert < \frac{1}{6\tau_0\gamma_0}$ have an ASIE.
\end{theomain}

\begin{rema}
This is very close to the estimate of Baumg\"artel \cite{Baumgartel}, see page 322 and further. However Baumg\"artel assumes $\fX$ is a Hilbert space; it might be possible to extend the method he employs to general Banach spaces, but the level of technicality makes it tedious to check. 

It is not easy to compare with the result of \cite{Nair} in general, notably because our choice of balance between precision and simplicity is slightly different. When $\tau_0=1$, $\lVert(\op{L}_0-\lambda_0)^{-1}\rVert=:1/\delta_0$ and $\lVert\pi_0\rVert=2$ (which is not uncommon, see Remark \ref{rema:Perron}), in the worst case Nair gets a radius of $\delta_0/16$ while we get $\delta_0/12$.
%
%Our main selling points compared to such previous results is the simplicity of the method (whose core is concentrated in Sections \ref{sec:estimates} and \ref{sec:proofs}) and the regularity estimates below.
\end{rema}

\begin{rema}\label{rema:Perron}
A toy application consists in applying Theorem \ref{theo:estimate1} in 
$\fX=\mathbb{R}^n$ with the supremum norm $\lVert\cdot\rVert_\infty$, to (a multiple of) the matrix $\op{L}_0$ having all coefficients equal to $1/n$, 
yielding the following. 

A matrix $\op{L}=(\ell_{ij})_{ij}$ that has almost constant coefficients in the sense that for some $c$, on all rows $i$ it holds 
\begin{equation}
\frac{1}{n}\sum_k \lvert \ell_{ik}-c\rvert\le \frac{\lvert c\rvert}{12},
\label{eq:Perron}
\end{equation}
must have a simple eigenvalue (here the $12$ comes from $\tau_0=1$ and $\gamma_0\le 2$). Under a slightly stronger bound, Theorem \ref{theo:estimate1} will also imply that the eigenvalue is positive, and we could further find conditions ensuring the eigenvector is positive too. 

This can be seen as a variation on the Perron-Frobenius Theorem since \eqref{eq:Perron} is fulfilled whenever for all coefficients $\lvert \ell_{ij}-c\rvert< \lvert c\rvert/12$ (the Perron-Frobenius Theorem would ask this with $1/12$ replaced by $1$, taking $c$ as the middle of the range interval of the coefficients); but \eqref{eq:Perron} is more flexible in that it allows for coefficients of variable sign (a small proportion of the coefficients can be very far from $c$).
\end{rema}

%%%%%%%%%%%%%%%%%%%%%%%%%%%%%%%%%%%%%%%%%%%%%%%%%%%%%%%%%%%%%%%%%
\subsection{Regularity estimates}

Next, at any distance smaller than our radius estimate we obtain effective regularity estimates. This is the main result of this article, to be used intensively in \cite{K:concentration}.
\begin{theomain}\label{theo:estimate1}
Given any $K>1$, whenever $\displaystyle \lVert\op{L}-\op{L}_0\rVert \le \frac{K-1}{6K\tau_0\gamma_0}$  we have
\begin{align*}
\lVert D\lambda\rVert &\le\tau_0+\frac{K-1}{3} & \lVert D\op{P}_\op{L} \rVert &\le 2K\tau_0\gamma_0 \\
\lVert D^2\lambda\rVert &\le 2K\tau_0\gamma_0  &\lVert D\pi_\op{L} \rVert &\le 2K\tau_0\gamma_0\\
\lVert D^3\lambda \rVert &\le 12 K^2 \tau_0^2\gamma_0^2, & &
\end{align*}
and the following Taylor formulas with explicit bounds:
\begin{align*}
\lambda_\op{L} &= \lambda_0 +O_{\tau_0+\frac{K-1}{3}} \big(\lVert\op{L}-\op{L}_0\rVert\big) \\
\lambda_\op{L} &= \lambda_0 + \phi_0(\op{L}-\op{L}_0)u_0 + O_{K \tau_0\gamma_0}\big(\lVert\op{L}-\op{L}_0\rVert^2\big) \\
\lambda_\op{L} &= \lambda_0 + \phi_0(\op{L}-\op{L}_0)u_0 - \phi_0(\op{L}-\op{L}_0)\op{S}_0(\op{L}-\op{L}_0)u_0 + O_{2 K^2\tau_0^2\gamma_0^2}\Big(\lVert \op{L}-\op{L}_0\rVert^3 \Big)\\
\op{P}_\op{L} &= \op{P}_0+O_{2K\tau_0\gamma_0}(\lVert\op{L}-\op{L}_0\rVert) \\
\pi_\op{L} &= \pi_0+O_{\tau_0+\frac{K-1}{3}}(\lVert\op{L}-\op{L}_0\rVert).
\end{align*}
\end{theomain}

\begin{rema}\label{rema:BE}
We stopped our estimates at order $3$ while it is easy (but slightly tedious) to use our methods up to any finite order, notably Proposition \ref{prop:derivatives} is easily extended. Our motivation to go precisely this far is in Berry-Esseen bounds: in \cite{K:concentration} we apply these estimates to Markov chains, seen as averaging operators on a suitable space of functions. Under a natural spectral gap assumption, the order $1$ term gives a law of large number and the order $1$ Taylor formula gives effective estimates in the convergence speed; the order $2$ Taylor development gives a Central Limit Theorem and the order $2$ Taylor formula gives effective estimate in the convergence speed, notably Berry-Esseen bounds.
\end{rema}

\begin{rema}
Expressed in terms of $r=\lVert\op{L}-\op{L}_0\rVert$, these bounds are
\begin{align*}
\lVert D\lambda\rVert &\le\tau_0+\frac{2\tau_0\gamma_0 r}{1-6\tau_0\gamma_0r} \\
\lVert D^2\lambda\rVert,\ \lVert D\op{P}_\op{L}\rVert,\ \lVert D\pi_\op{L} \rVert &\le \frac{2\tau_0\gamma_0}{1-6\tau_0\gamma_0 r}\\
\lVert D^3\lambda \rVert &\le \frac{12 \tau_0^2\gamma_0^2}{(1-6\tau_0\gamma_0 r)^2}.
\end{align*}
\end{rema}

%\begin{rema}
%For simplicity we restricted to bounded operators, but using the graph norm on its domain the case when $\op{L}_0$ is only closed can be treated similarly, as long as the perturbations are bounded (i.e. $\op{L}-\op{L}_0\in\BX$).
%\end{rema}

%\begin{rema}
%In many concrete cases the formulas above simplify and are easily interpreted. Let us give a simple example for the heat flow. Let $(M,g)$ be a compact Riemannian manifold and $\op{L}_0=e^\Delta$ acting on, say, the Sobolev space $W^{1,2}(M)$ of $L^2$ functions with $L^2$ derivatives. Of course one usually considers the flow $e^{t\Delta}$, but its long-time behavior can be read on the operator $\op{L}_0$.
%\end{rema}

%%%%%%%%%%%%%%%%%%%%%%%%%%%%%%%%%%%%%%%%%%%%%%%%%%%%%%%%%%%%%%%%%%%%%%%%
\subsection{Spectral gap estimates}

In some applications, we have more than an isolated eigenvalue: a spectral gap below $\lambda_0$. It is well-known that the operators having a spectral gap form an open set, and we shall provide a quantitative version of this statement.
\begin{defi}\label{defi:gap}
We shall say that $\op{L}\in\BX$ has a \emph{spectral gap} (\emph{of size $\delta\in(0,1)$} \emph{with constant $C\ge 1$}) below its eigenvalue $\lambda$ if on the stable complement $G$ to the one-dimensional eigenspace it holds
\[\lVert \op{L}^n x \rVert \le C \lvert\lambda\rvert^n(1-\delta)^n \lVert x\rVert \quad\forall x\in G,  \forall n\in\mathbb{N}.\]
\end{defi}
Under the assumption of a spectral gap, $\lambda$ is not only isolated from the rest of the spectrum: the rest of the spectrum is contained in a disc of radius $\lvert\lambda\rvert(1-\delta)$. We shall then call $\lambda$ the leading (or main) eigenvalue.

When it comes to perturbations, the simplest case to handle is when $C=1$, i.e. $\frac{1}{\lambda_0}\op{L}_0$ is contracting on $G_0$.
\begin{theomain}\label{theo:spectral-gap}
Assume $\op{L}_0$ has a spectral gap of size $\delta_0$ below its leading eigenvalue $\lambda_0$ with constant $C_0=1$, i.e.
\[\lVert \op{L}_0 x\rVert\le (1-\delta_0)\lvert\lambda_0\rvert \lVert x\rVert \quad \forall x\in G_0.\]
Set $a=2\big(\lvert\lambda_0\rvert(1-\delta_0)+\lVert\op{L}_0\rVert\big)$.
Given $\delta\in(0,\delta_0)$, let $\rho(\delta)$ be the unique positive root of
\[X^2\Big(a+\frac{1-\delta}{6\tau_0\gamma_0} \Big) +X\Big(6\lvert\lambda_0\rvert(\delta-\delta_0) + a+\frac{1-\delta}{\gamma_0}+\frac{1}{\tau_0\gamma_0} \Big)+6\lvert\lambda_0\rvert (\delta-\delta_0).\]
Then every $\op{L}\in\BX$ such that
\[\lVert\op{L}-\op{L}_0\rVert \le \frac{\rho(\delta)}{6(1+\rho(\delta))\tau_0\gamma_0}\]
has a spectral gap of size $\delta$ below $\lambda_\op{L}$, with constant $1$.
\end{theomain}
Note that $\rho(\delta)$ tends to $0$ as $\delta\to\delta_0$ and has a finite limit when $\delta\to 0$, which gives a lower bound on the radius around $\op{L}_0$ where some spectral gap persists. The expressions are a bit intricate, but they only depend on the numerical quantities $\tau_0$, $\gamma_0$, $\lVert L_0\rVert$, $\lvert\lambda_0\rvert$, $\delta_0$, neither on the specific value of $\op{L}_0$ nor on any property of $\fX$.

Under quite common further assumptions, we can simplify the result if we accept to loose some precision.
\begin{coromain}\label{coro:gap-short}
In the case $\lambda_0=\lVert \op{L}_0\rVert=C_0=1$, all $\op{L}$ such that
\[\lVert\op{L}-\op{L}_0\rVert \le \frac{\delta_0(\delta_0-\delta)}{6(1+\delta_0-\delta)\tau_0\lVert\pi_0\rVert}\]
have a spectral gap of size $\delta$ below $\lambda_\op{L}$, with constant $1$. In particular, all $\op{L}$ such that
\[\lVert\op{L}-\op{L}_0\rVert < \frac{\delta_0^2}{6(1+\delta_0)\tau_0\lVert\pi_0\rVert}\]
have some spectral gap, with constant $1$.
\end{coromain}

\begin{rema}
While the assumptions may seem quite restrictive, they are relevant to the case when $\op{L}$ belong to a family of ``transfer operators'' associated with various potentials for a fixed dynamical system; we apply Corollary \ref{coro:gap-short} to this context in \cite{K:HT}.
\end{rema}

The case when $C_0>1$ is technically more involved. Instead of working out the numbers, we simply state a uniform but non-effective result.
\begin{coro}\label{coro:gap-ne}
If $\op{L}_0$ has a spectral gap of size $\delta_0$ with constant $C_0$ below its eigenvalue $\lambda_0$, then all $\op{L}$ such that
\[\lVert\op{L}-\op{L}_0\rVert \le O_{*C_0,\delta_0^{-1},\tau_0,\lvert\lambda_0\rvert}(1)\]
have a spectral gap below $\lambda_\op{L}$.
\end{coro}

\begin{rema}
Continuity of eigenvalues has been proved by Keller and Liverani \cite{Keller-Liverani} for more general perturbations. More specifically, they considered (under a specific set of assumptions) the case when we have an additional weaker (not complete) norm $\lVert\cdot\rVert_w$ on $\fX$ 
and the perturbation is small in the strong-to-weak operator norm
\[\lVert \op{L} \rVert_{sw} := \sup\{\lVert Lx\rVert_w : \lVert x\rVert\le 1 \}.\]
It would be interesting to see whether radius bounds and regularity estimates as above can be derived in this setting, which is notably important in dynamical systems (for example, the perturbation induced on the ``transfer operator'' of a perturbed dynamical system of hyperbolic type is often large in the usual operator norm, but small in the strong-to-weak norm).
\end{rema}

%%%%%%%%%%%%%%%%%%%%%%%%%%%%%%%%%%%%%%%%%%%%%%%%%%%%%%%%%%%%%%%
%%%%%%%%%%%%%%%%%%%%%%%%%%%%%%%%%%%%%%%%%%%%%%%%%%%%%%%%%%%%%%%
\section{Prerequisites}\label{sec:prer}

%%%%%%%%%%%%%%%%%%%%%%%%%%%%%%%%%%%%%%%%%%%%%%%%%%%%%%%%%%%%%%%
\subsection{Analyticity in Banach spaces}\label{sec:analytic}

Analyticity in Banach spaces is very similar to analyticity on $\mathbb{R}$ or $\mathbb{C}$, but for the sake of completeness let us recall the definition and a few properties. Note that the definition we give is a strong one, some authors only asking for composition with analytic paths to be analytic. This weaker definition gives no uniformity with respect to the direction of a perturbation and is thus not suitable for our present purpose.

Let $\fspace{X}$ and $\fspace{Y}$ be two (real or complex) Banach spaces,
whose norms will both be denoted by $\lVert\cdot\rVert$.
A continuous, symmetric, multilinear operator
$\xi: \fspace{X}^k \to \fspace{Y}$ has an operator
norm denoted by $\lVert \xi\rVert$; if $x$ is a vector
in $\fspace{X}$, we set $\xi(x):=\xi(x,x,\dots,x)$
and we have $\lVert \xi(x)\rVert\le \lVert\xi\rVert \lVert x\rVert^k$.
We shall say that a sequence
$\xi_k:\fspace{X}^k \to \fspace{Y}$ of continuous symmetric $k$-linear operators
($k\ge 0$) is a series with positive radius of convergence
if the complex series
\[\sum_{k\ge 0} \lVert \xi_k\rVert z^k \]
has a positive radius of convergence in $\mathbb{C}$.

Let $F:U\subset\fspace{X}\to\fspace{Y}$ be a map defined
on an open subset of $\fspace{X}$. We say that $F$
is \emph{analytic} if for each $x\in U$ there is a series
of $k$-linear, symmetric, continuous operators
$\xi_{x,k}:\fspace{X}^k \to \fspace{Y}$ with positive
radius of convergence such that the following identity holds
for all $h$ in a neighborhood of the origin in $\fX$:
\begin{equation}
F(x+h) = \sum_{k\ge 0} \xi_{x,k}(h)
\label{eq:analytic}
\end{equation}
(note that as soon as $\lVert h \rVert$ is small enough, the sum is absolutely convergent, hence convergent).

An analytic map is smooth (in particular, Fr\'echet differentiable
and locally Lipschitz-continuous) and the operators $\xi_{x,k}$
are uniquely defined by $F$. Moreover it suffices to check 
\eqref{eq:analytic} at a point $x$ to have a similar expansion
$F(y+h)= \sum \xi_{y,k}(h)$ for all $y$ in a neighborhood of $x$.

%%%%%%%%%%%%%%%%%%%%%%%%%%%%%%%%%%%%%%%%%%%%%%%%%%%%%%%%%%%%%%%
\subsection{The Implicit Function Theorem}\label{sec:IFT}

The implicit function theorem is well-known for smooth maps between finite-dimensional spaces; it holds as well in the analytic regularity, for maps between Banach spaces, with basically the same proof (see e.g. \cite{Chae}, \cite{Whi}).
% The only adaptation is that to be able to identify the starting space $\fX$ with a direct product of the kernel of the differential $\ker DF_{x}$ and the target space $\fspace{Y}$, we need that the kernel be complemented (this is automatic in finitely many dimensions).

\begin{theo*}[Implicit function theorem]
Let $F:U\subset\fspace{X}\times\fspace{Y}\to\fspace{Z}$ be an analytic map defined on an open set of the product space, such that $F(x_0,y_0)=0$ for some $(x_0,y_0)\in U$. If $\partial_2 F_{(x_0,y_0)}:\fspace{Y}\to\fspace{Z}$ is a linear isomorphism of Banach spaces, then there is an analytic map
$Y:\fspace{V}\to \fspace{Y}$ defined in a neighborhood $\fspace{V}$ of $x_0$ such that $Y(x_0)=y_0$ and for all $x\in\fspace{V}$, $F(x,Y(x))=0$.
Moreover for each $x$ close enough to $x_0$, $y=Y(x)$ is the only solution to $F(x,y)=0$ in a neighborhood of $y_0$.
\end{theo*}

Here $\partial_2$ denotes the second partial derivative of $F$, i.e.
$\partial_2 F_{(x,y)}(h) = DF_{(x,y)}(0,h)$. 

\begin{rema}
To treat the case when $\op{L}_0$ is closed rather than bounded, one needs a more general Implicit Function Theorem, suitable for a map of the form
$F(x,y)=Ay +b(x,y)$ where $b$ is analytic and $A$ is linear and closed with domain $\fspace{D}\subset \fspace{Y}$. Such an Implicit Function Theorem is easily deduced from the above one by endowing $\fspace{D}$ with the graph norm $\max(\lVert\cdot\rVert,\lVert A\cdot \rVert)$ making it a Banach space. Then $Ay+b(x,y)$ defines an analytic map from an open set of $\fspace{X}\times\fspace{D}$ and the above theorem yields an implicit function $Y:\fspace{V}\subset\fspace{X}\to \fspace{D}$, which is still analytic when seen with target $\fspace{Y}$.

The key point to observe is that in the proof of analyticity of the eigendata we can replace the factor $\fX\times\mathbb{K}$ in the source by $\fspace{D}\times\mathbb{K}$, but let the target be $\fX\times\mathbb{K}$. Then we need that $\op{L}_0-\lambda_0$ be invertible from $\fspace{D}\cap G_0$ to $G_0$, with \emph{bounded} inverse (which is the usual hypothesis).
\end{rema}

%%%%%%%%%%%%%%%%%%%%%%%%%%%%%%%%%%%%%%%%%%%%%%%%%%%%%%%%%%%%%%%
\subsection{Metric derivative}\label{sec:metric}

Our effective estimates are obtained by controling the evolution of the quantities $\tau$ and $\gamma$ when the operator $\op{L}$ moves away from $\op{L}_0$. We shall use differential inequalities to compare $\tau$ and $\gamma$ to the solution of a system of ODE, with the slight complication that $\tau$ and $\gamma$ are not differentiable. We shall rely on the simple notion of the \emph{metric derivative}, also named \emph{pointwise Lipschitz constant}, of a function $f:\fY\to \mathbb{K}$, which we denote by $\llip$:
\[\llip f (x) := \limsup_{r\to 0} \sup_{y\in B(x,r)}\frac{\lvert f(x)-f(y)\rvert}{\lVert x-y\rVert}.\]
Of course, if $f$ is (Fréchet) differentiable then $\llip f=\lVert D f\rVert$, and if $f$ is $C$-Lipschitz then $\llip f\le C$.

A way to reword this definition is by saying that $\llip f(x)$ is the least constant $C$ such that
\[\lvert f(x)- f(y)\rvert \le C \lVert x-y\rVert +o(\lVert x-y\rVert) \quad\mbox{as } y\to x. \]
This makes it easy to check that for all (locally Lipschitz, say) functions $f,g, (f_i)_{i\in I}:\fY\to\mathbb{K}$ we have
\begin{align*}
\llip \lvert f\rvert &\le \llip f  \\
\llip (fg) &\le (\llip f) \lvert g\rvert + \lvert f\rvert(\llip g) \\
\llip \sup_{i\in I}(f_i) &\le \sup_{i\in I}(\llip f_i)
\end{align*}
and if $f$ takes it values in a Banach space and is differentiable,
\[\llip \lVert f\rVert \le \lVert Df\rVert.\]

Moreover, we have the usual comparison result for differential inequalities (which we state here in a version which is easy to prove and sufficient for our purpose, but certainly less general than possible).
\begin{prop}\label{prop:comparison}
Let $f:\fspace{V}\subset \fY\to\mathbb{R}$ be a function defined on a convex open set of a Banach space and $F: [0,+\infty) \to [0,+\infty)$ be a non-decreasing locally Lipschitz function.
Fix $x_0\in\fspace{V}$ and let $s:[0,R)\to \mathbb{R}$ (with $R$ finite) be the solution to $\big(s'=F(s), s(0)=\lvert f(x_0)\rvert\big)$.

If $\llip f(x)\le F(\lvert f(x)\rvert)$ for all $x\in \fspace{V}$, then for all 
$x\in \fspace{V}\cap B(x_0,R)$ we have
\[ \lvert f(x)\rvert \le s(\lvert x-x_0\rvert).\]
\end{prop}

\begin{proof}
We first restrict ourselves to dimension $1$. Let $x\in \fspace{V}$ such that $\lVert x-x_0\rVert <R$ and consider $x_t=x_0+t(x-x_0)/\lVert x-x_0\rVert$ and $g(t)=\lvert f(x_t)\rvert$
for $t\in [0,R)$. Then $\llip g(t) \le \llip f(x_t) \le F(g(t))$.

Let $C$ be a Lipschitz constant for $F$, valid on $[0,R]$. Given $\varepsilon>0$, consider the set
\[A=\big\{ t_0\in [0,R) \,\big|\, \forall t\le t_0: g(t) \le s(t) + \varepsilon e^{2Ct}\big\}.\]
We have $0\in A$ by the initial data imposed on $s$, and $A$ is clearly an interval closed in $[0,R)$. Assume $T:=\sup A<R$; then we have 
$g(T) \le s(T) +\varepsilon e^{2CT}$ and thus
\[\llip g(T)\le F(g(T)) \le F(s(T)+\varepsilon e^{2CT})\le s'(T) + C\varepsilon e^{2CT}.\]
For $t\to T$ and $t>T$, taking the difference between
\[g(t)\le s(T)+\varepsilon e^{2CT} + (s'(T)+C\varepsilon e^{2CT})(t-T) + o(t-T)\]
and
\[s(t)+\varepsilon e^{2Ct} =  s(T)+\varepsilon e^{2CT} + (s'(T)+2C\varepsilon e^{2CT})(t-T) + o(t-T) \]
we get
\[g(t)\le s(t)+\varepsilon e^{2Ct} -Ce^{2CT}(t-T)+ o(t-T).\]
Therefore there exists a $T'>T$ such that $g(t) \le s(t)+\varepsilon e^{2Ct}$ for $t\in [0,T')$, contradicting the definition of $T$. Thus $T=R$ and for all $t\in [0,R)$ and all $\varepsilon>0$ we have $g(t) \le s(t) + \varepsilon e^{2Ct}$. Passing to the limit when $\varepsilon\to 0$, we get the desired conclusion.
\end{proof}

%%%%%%%%%%%%%%%%%%%%%%%%%%%%%%%%%%%%%%%%%%%%%%%%%%%%%%%%%%%%%%%
%%%%%%%%%%%%%%%%%%%%%%%%%%%%%%%%%%%%%%%%%%%%%%%%%%%%%%%%%%%%%%%
\section{Derivatives of the eigendata}\label{sec:derivatives}

We fix a bounded operator $\op{L}_0$ defined on $\fX$ to itself having a simple isolated eigenvalue $\lambda_0$, an eigenvector $u_0$ and an eigenform $\phi_0$. For simplicity, we shall assume that $\phi_0u_0=1$. This has no incidence on statements and quantities which are invariant under changing this normalization, such as estimates on $\lambda$, the value of $\tau$, etc. Other cases can be recovered by homogeneity considerations if necessary.

%We denote by $\op{S}_0$ the reduced resolvent of $\op{L}_0$, i.e.
%\[\op{S}_0=(\op{L}_0-\lambda_0)^{-1} \pi_0 \in \BX\]
%where $G_0=\ker\phi_0$ is the stable complement to $\langle u_0\rangle$ and $\pi_0=\Id-\phi_0(\cdot)u_0$ the projection to $G_0$ along $u_0$.

%%%%%%%%%%%%%%%%%%%%%%%%%%%%%%%%%%%%%%%%%%%%%%%%%%%%%%%%%%%%%%%
\subsection{The perturbed eigenvalue is simple isolated}

Our starting point is the qualitative theorem stated and proved in the introduction, according to which there are analytic maps $\lambda, u$ defined in a neighborhood $\fspace{V}$ of $\op{L}_0$ in $\fX$, with values in $\mathbb{K}$ and $\fX$ respectively, such that $\op{L}u_\op{L}=\lambda_\op{L} u_\op{L}$; moreover $\lambda_\op{L}$ is the only eigenvalue of $\op{L}$ near $\lambda_0$. Up to further restrictions we assume $\fspace{V}$ to be star-shaped with respect to $\op{L}_0$.

We moreover apply the same result to $\op{L}_0^*$ to obtain an analytic map $\phi:\fspace{V}\to \fX^*$ such that $\phi_\op{L}$ is an eigenvector of $\op{L}^*$ for an eigenvalue that is close to $\lambda_0$ and must thus be $\lambda_\op{L}$. Then $\ker\phi_\op{L}=: G_\op{L}$ is a closed hyperplane which is $\op{L}$-invariant and complementary to both $\langle u_\op{L}\rangle$ and $\langle u_0\rangle$. The hyperplanes
$G_\op{L}$ and $G_0$ can then be identified through $\pi_0$. Since $\op{L}-\lambda_\op{L}$ is close to $\op{L}_0-\lambda_0\in\fspace{B}(G_0)$ through this identification, it must be invertible. This proves the following classical strengthening of the qualitative theorem stated in the introduction.
\begin{prop}\label{prop:open}
Each $\op{L}$ in some neighborhood $\fspace{V}$ of $\op{L}_0$ (which has possibly been further reduced) has $\lambda_\op{L}$ as \emph{simple isolated} eigenvalue.
\end{prop}

We insist on the difference between the two eigendata $\lambda_\op{L}$ and $u_\op{L}$: $\lambda$ is completely specified, while $u$ is subject to normalization, as for every analytic function $f:\fspace{V}\to \mathbb{K}$, the map $e^fu$ also defines a eigenvector. Similarly, $\phi_\op{L}$ can be replaced by $e^g\phi$ freely. We shall assume that $u$ is as constructed above (i.e. $\phi_0 u_\op{L} \equiv 1$) but rescale $\phi_\op{L}$
to enforce the relation
\begin{equation*}
\phi_\op{L} u_\op{L} = 1 \quad \forall \op{L}\in\fspace{V}
\end{equation*}
which we assume from now on.

%%%%%%%%%%%%%%%%%%%%%%%%%%%%%%%%%%%%%%%%%%%%%%%%%%%%%%%%%%%%%%
\subsection{First derivative of the eigenvalue and the eigenvector}

As is classical when one uses the Implicit Function Theorem, the derivatives of the implicit function can be recovered by differentiating $F(\op{L},u_\op{L},\lambda_\op{L})\equiv 0$, yielding the following.

\begin{prop}\label{prop:DU}
On $\fspace{V}$ we have 
\[D\lambda_\op{L}(\op{M}) = \phi_\op{L}\op{M} u_\op{L} \qquad \forall \op{M}\in\BX\]
and there is an analytic map $a:\fspace{V}\to \BX^*$ vanishing at $\op{L}_0$ such that at all $\op{L}\in\fspace{V}$:
\[Du_\op{L}(\op{M})=-\op{S}_\op{L} \op{M}u_\op{L} + a_\op{L}(\op{M}) u_\op{L}.\]
\end{prop}

\begin{proof}
Recall that $u$ and $\lambda$ are obtained by applying the Implicit Function Theorem to the map
\begin{align*}
F : \BX \times (\fX \times \mathbb{K}) &\to \fX \times \mathbb{K} \\
  (\op{L},u,\lambda) &\mapsto (\op{L}u-\lambda u,\, \phi_0(u)-1).
\end{align*}
Differentiating 
$F(\op{L},u_\op{L},\lambda_\op{L})\equiv 0$ in a direction $\op{M}\in\BX$ we obtain
\[\Big((\op{L}-\lambda_\op{L}) Du_{\op{L}}(\op{M})+\op{M} u_\op{L}-D\lambda_{\op{L}}(\op{M}) u_\op{L},\, \phi_0(Du_{\op{L}}(\op{M}))\Big) = 0.\]
Applying $\phi_\op{L}$ to the first member and using $\phi_\op{L}(\op{L}-\lambda_\op{L})=0$, we get $D\lambda_{\op{L}}(\op{M}) = \phi_\op{L}\op{M}u_\op{L}$.
 Applying $S_\op{L}$ to the first member and using $\pi_\op{L}(\op{L}-\lambda_\op{L})=(\op{L}-\lambda_\op{L})\pi_\op{L}$ we get $\pi_\op{L}Du_{\op{L}}(\op{M})=-\op{S}_\op{L}\op{M}u_\op{L}$, and setting $a_\op{L}(\op{M}):=\phi_\op{L}(Du_\op{L}(\op{M}))$ we are done.
\end{proof}

\begin{rema}
In the general case where we do not assume the normalization $\phi_\op{L}u_\op{L}\equiv 1$, by invariance with respect to normalization we have
\[D\lambda_\op{L}(\op{M}) = \frac{\phi_\op{L}\op{M} u_\op{L}}{\phi_\op{L} u_\op{L}}.\]
\end{rema}

\begin{rema}
We have
$D\big[e^fU\big]_\op{L}(\op{M})=Df_\op{L}(\op{M}) e^{f_\op{L}}u_\op{L} + e^{f_\op{L}} Du_\op{L}(\op{M})$
so that, for any fixed $\op{L}$, by choosing $f$ such that $Df_\op{L}(\cdot)=-\phi_\op{L}(Du_\op{L}(\cdot))$ we can ensure $D\big[e^fU\big]_\op{L}(\op{M})\in G_\op{L}$ for all $\op{M}$. However we may not be able to ensure this property simultaneously at all $\op{L}$, because we would need $\op{L}\mapsto-\phi_\op{L}(Du_\op{L}(\cdot))$ to be a closed $1$-form. 
\end{rema}

%%%%%%%%%%%%%%%%%%%%%%%%%%%%%%%%%%%%%%%%%%%%%%%%%%%%%%%%%%%%%%
\subsection{First derivative of the eigenform}

Let $G_\op{L}^*=u_\op{L}^\perp$ be the stable complement of $\langle\phi_\op{L}\rangle$ for $\op{L}^*$ and $\pi_\op{L}^*$ the corresponding projection.  Note that $\pi_\op{L}^*$ also happens to be the dual operator to $\pi_\op{L}$, and that as before $(\op{L}^*-\lambda_\op{L})^{-1}$ is by convention a map from $G_\op{L}^*$ to itself.

\begin{lemm}\label{lemm:subt}
For all $\psi\in\fX^*$, we have
\[ (\op{L}^*-\lambda_\op{L})^{-1}\pi_\op{L}^*\psi = \psi(\op{L}-\lambda_\op{L})^{-1}\pi_\op{L}\]
i.e. $\op{S}_\op{L}^*=\op{S}_{\op{L}^*}$.
As a consequence, $\gamma_{\op{L}^*} = \gamma_\op{L}$.
It also holds $\tau_{\op{L}^*}=\tau_\op{L}$.
\end{lemm}

The order of composition with $(\op{L}-\lambda_\op{L})$ and $\pi_\op{L}$ may seem wrong, but this is a subtlety in the definitions related to the domain of $(\op{L}^*-\lambda_\op{L})^{-1}$ (note that the other order of composition would not make sense).

\begin{proof}
We have $(\op{L}-\lambda_\op{L})^{-1}\pi_\op{L}=\pi_\op{L}(\op{L}-\lambda_\op{L})^{-1}\pi_\op{L}$ (and similarly in the dual) so that
\[(\op{L}^*-\lambda_\op{L})^{-1}\pi_\op{L}^*\psi = \pi_\op{L}^*(\op{L}^*-\lambda_\op{L})^{-1}\pi_\op{L}^*\psi = \psi\circ\big(\pi_\op{L}(\op{L}-\lambda_\op{L})^{-1}\pi_\op{L}\big) = \psi(\op{L}-\lambda_\op{L})^{-1}\pi_\op{L}.\]
We deduce $\gamma_{\op{L}^*}\le \gamma_\op{L}$, and the equality follows by using the Hahn-Banach Theorem to find a $\psi$ whose norm is realized by $(\op{L}-\lambda_\op{L})^{-1}\pi_\op{L}(x)$ where $x$ almost realizes $\gamma_\op{L}$.

Last, since the natural image of $u_\op{L}$ in the bidual $\fX^{**}$ is obviously the eigenvector of $\op{L}^{**}$ for $\lambda_\op{L}$, we have $\tau_{\op{L}^*}=\tau_\op{L}$.
\end{proof}

\begin{prop}
For all $\op{L}\in\fspace{V}$ it holds
\[D\phi_\op{L}(\op{M})=-\phi_\op{L}\op{M}\op{S}_\op{L}- a_\op{L}(\op{M}) \phi_\op{L}\]
where $a$ is the $1$-form on $\fspace{V}\subset \BX$ defined in Proposition \ref{prop:DU}.
\end{prop}

\begin{proof}
Applying Proposition \ref{prop:DU} to $\op{L}^*$, there must be an analytic map $b:\fspace{V}\to \BX^*$ such that 
\[D\phi_\op{L}(\op{M}) =(\lambda_\op{L}-\op{L}^*)^{-1} \pi_\op{L}^*(\phi_\op{L}  \op{M}) +b_\op{L}(\op{M})\phi_\op{L}.\]
Lemma \ref{lemm:subt} allows us to rewrite this as
\[D\phi_\op{L}(\op{M})= \phi_\op{L}  \op{M}  (\lambda_\op{L}-\op{L})^{-1}  \pi_\op{L}+b_\op{L}(\op{M})\phi_\op{L}= -\phi_\op{L}  \op{M} \op{S}_\op{L}+b_\op{L}(\op{M})\phi_\op{L}.\]
Differentiating $1 \equiv \phi_\op{L} u_\op{L}$ and using that $\op{S}_\op{L}u_\op{L}=0$ and that $\phi_\op{L}$ vanishes on the range $G_\op{L}$ of $\op{S}_\op{L}$ we then get
\begin{align*}
0 &=D(\phi u)_\op{L}(\op{M}) \\
 & = \big[-\phi_\op{L}  \op{M}  \op{S}_\op{L}+b_\op{L}(\op{M})\phi_\op{L}\big]u_\op{L} + \phi_\op{L}\big[-\op{S}_\op{L} \op{M}u_\op{L} + a_\op{L}(\op{M}) u_\op{L}\big] \\
 &= b_\op{L}(\op{M})\phi_\op{L} u_\op{L}+\phi_\op{L}(a_\op{L}(\op{M}) u_\op{L}).
\end{align*}
It follows $0=b_\op{L}(\op{M})+a_\op{L}(\op{M})$.
\end{proof}

%%%%%%%%%%%%%%%%%%%%%%%%%%%%%%%%%%%%%%%%%%%%%%%%%%%%%%%%%%%%%%
\subsection{Further differentiation formulas}

We will now compute derivatives of other quantities and higher derivatives of $\lambda$. Unsurprisingly, the normalizing function $a$ will often disappear: it cannot impact the quantities that are normalization-insensitive.

It is sometime useful to consider the operator $\op{R}_\op{L}$ defined by $\op{L}=\lambda_{\op{L}} \op{P}_{\op{L}} + \op{R}_{\op{L}}$. It takes its values in $G_\op{L}$, vanishes on $u_\op{L}$ and therefore satisfies $\op{P}_\op{L}\op{R}_\op{L} = \op{R}_\op{L}\op{P}_\op{L} = 0$. The spectral gap condition can then be rephrased as an exponential decay of $\lVert \op{R}_\op{L}^n\rVert$.

Recall that the normalization $\phi_\op{L} u_\op{L}\equiv 1$ is assumed; we also gather the previously computed derivatives to ease future reference.
\begin{prop}\label{prop:derivatives}
For all $\op{L}$ having an ASIE, we have the following expressions:
\begin{enumerate}
\item $D\lambda_\op{L}(\op{M}) =  \phi_\op{L}  \op{M}  u_\op{L}$,
\item $Du_\op{L}(\op{M})=-\op{S}_\op{L} \op{M} u_\op{L} +a_\op{L}(\op{M}) u_\op{L}$,
\item $D\phi_\op{L}(\op{M})=- \phi_\op{L} \op{M} \op{S}_\op{L} -a_\op{L}(\op{M}) \phi_\op{L}$,
\item $D^2\lambda_\op{L}(\op{M}) = -2 \phi_\op{L}  \op{M}  \op{S}_\op{L} \op{M} u_\op{L}$,
\item $D\op{P}_\op{L}(\op{M}) =  -\phi_\op{L}  \op{M} \op{S}_\op{L}(\cdot) u_\op{L}-\phi_\op{L}(\cdot)\op{S}_\op{L}Mu_\op{L}$,  
\item $D\pi_\op{L}(\op{M}) = \phi_\op{L}  \op{M} \op{S}_\op{L}(\cdot) u_\op{L}+\phi_\op{L}(\cdot)\op{S}_\op{L}Mu_\op{L}$,
\item $D\op{S}_\op{L}(\op{M}) = -\op{S}_\op{L} \op{M}\op{S}_\op{L}(\cdot)+\phi_\op{L}(\cdot)\op{S}_\op{L}^2\op{M}u_\op{L} +(\phi_\op{L}\op{M}u_\op{L})\op{S}_\op{L}^2 +\big[\phi_\op{L}\op{M}\op{S}_\op{L}^2(\cdot)\big]u_\op{L}$,
\item $D^3\lambda_\op{L}(\op{M})=6\phi_\op{L}  \op{M}  \op{S}_\op{L}\big[\op{M}-\phi_\op{L}\op{M} u_\op{L}\big] \op{S}_\op{L} \op{M} u_\op{L}$,
\item as soon as $\lambda_\op{L}\neq0$, $D\big[\frac{1}{\lambda} \op{R}\big]_\op{L}(\op{M})=\frac{1}{\lambda_\op{L}}\op{M}-\frac{\phi_\op{L}\op{M} u_\op{L}}{\lambda_\op{L}^2}\op{L}+\phi_\op{L}  \op{M} \op{S}_\op{L}(\cdot) u_\op{L}+\phi_\op{L}(\cdot)\op{S}_\op{L}Mu_\op{L}$.
\end{enumerate}
\end{prop}

In the bracket of the second-to-last item, the scalar $\phi_\op{L}\op{M} u_\op{L}$ is to be interpreted as the scalar operator $(\phi_\op{L}\op{M} u_\op{L}) \Id$.

\begin{proof}
The first three items have been proved above.
Differentiating $D\lambda_\op{L}(\op{M})=\phi_\op{L}\op{M} u_\op{L}$ we get:
\begin{align*}
D^2\lambda_\op{L}(\op{M})
  &=\Big(-\phi_\op{L}\op{M}\op{S}_\op{L}-a_\op{L}(\op{M})\phi_\op{L}\Big)\op{M}u_\op{L}+\phi_\op{L}\op{M} \Big(-\op{S}_\op{L}\op{M}u_\op{L}+a_\op{L}(\op{M})u_\op{L}\Big)\\
  &= -2\phi_\op{L}\op{M}\op{S}_\op{L}\op{M}u_\op{L}-a_\op{L}(\op{M})\phi_\op{L}\op{M}u_\op{L}+\phi_\op{L}\op{M}(a_\op{L}(\op{M})u_\op{L})\\
   &=-2\phi_\op{L}\op{M}\op{S}_\op{L}\op{M}u_\op{L}.
\end{align*}
The formula for $D\op{P}$ is obtained similarly by differentiating $\op{P}_\op{L}=\phi_\op{L}(\cdot)u_\op{L}$, with some caution: the terms $-a_\op{L}(\op{M})\phi_\op{L}(\cdot)u_\op{L}$ and $\phi_\op{L}(\cdot)a_\op{L}(\op{M})u_\op{L}$ do cancel out, but the two remaining terms 
$-\phi_\op{L}  \op{M} \op{S}_\op{L}(\cdot) u_\op{L}$ and $-\phi_\op{L}(\cdot)\op{S}_\op{L}Mu_\op{L}$ are quite different: the first one maps $x\in\fX$ to $-(\phi_\op{L}  \op{M} \op{S}_\op{L}x) u_\op{L}\in \langle u_\op{L}\rangle$,  while the second maps it to $-(\phi_\op{L}x)\op{S}_\op{L}Mu_\op{L}\in G_\op{L}$.
Differentiating $\pi_\op{L}=\Id-\op{P}_\op{L}$ with respect to $\op{L}$ (observe that $\Id$ is a constant), we get $D\pi_\op{L} = -D\op{P}_\op{L}$.

To treat $\op{S}$, one first differentiates $(\op{L}-\lambda_\op{L})\op{S}_\op{L} = \pi_\op{L}$ to get:
\[(\op{L}-\lambda_\op{L})D\op{S}_\op{L}(\op{M}) + (\op{M}-\phi_\op{L}\op{M}u_\op{L})\op{S}_\op{L}
  = \phi_\op{L}  \op{M} \op{S}_\op{L}(\cdot) u_\op{L} + \phi_\op{L}(\cdot)\op{S}_\op{L}Mu_\op{L} \]
and obtains
\begin{align*}
(\op{L}-\lambda_\op{L})D\op{S}_\op{L}(\op{M}) &= -\op{M}\op{S}_\op{L} +\phi_\op{L}  \op{M} \op{S}_\op{L}(\cdot) u_\op{L} +(\phi_\op{L}\op{M}u_\op{L})\op{S}_\op{L}+ \phi_\op{L}(\cdot)\op{S}_\op{L}Mu_\op{L}\\
  &= -\pi_\op{L} \op{M}\op{S}_\op{L} +(\phi_\op{L}\op{M}u_\op{L})\op{S}_\op{L}+ \phi_\op{L}(\cdot)\op{S}_\op{L} \op{M} u_\op{L}.
\end{align*}
Composing on the left by $\op{S}_\op{L}$ and observing that $\op{S}_\op{L}(\op{L}-\lambda_\op{L}) = \pi_\op{L}$ and $\op{S}_\op{L}\pi_\op{L}=\op{S}_\op{L}$ and using linearity to pull scalar expressions out of operator arguments, it follows that
\begin{equation}
\pi_\op{L}D\op{S}_\op{L}(\op{M}) = -\op{S}_\op{L} \op{M}\op{S}_\op{L} +(\phi_\op{L}\op{M}u_\op{L})\op{S}_\op{L}^2+ \phi_\op{L}(\cdot)\op{S}_\op{L}^2 \op{M} u_\op{L}.
\label{eq:DS1}
\end{equation}
Then one differentiates $\phi_\op{L}\op{S}_\op{L}\equiv 0$ to obtain
\[\phi_\op{L} D\op{S}_\op{L}(\op{M})=\phi_\op{L}\op{M}\op{S}_\op{L}^2,\]
giving the $u_\op{L}$ component of $D\op{S}_\op{L}(\op{M})$. Combining this information with \eqref{eq:DS1}, the claimed formula follows.

Last, differentiating $D^2\lambda_\op{L}(\op{M}) = -2 \phi_\op{L}  \op{M}  \op{S}_\op{L} \op{M} u_\op{L}$ we arrive at
\begin{align*}
D^3\lambda_\op{L}(\op{M}) &= -2D\phi_\op{L}(\op{M}) \op{M} \op{S}_\op{L} \op{M} u_\op{L} - 2\phi_\op{L} \op{M} D\op{S}_\op{L}(\op{M}) \op{M} u_\op{L}-2\phi_\op{L}  \op{M}  \op{S}_\op{L} \op{M} Du_\op{L}(\op{M}) \\
&=6\phi_\op{L}  \op{M}  \op{S}_\op{L} \op{M}  \op{S}_\op{L} \op{M} u_\op{L}-6(\phi_\op{L} \op{M}\op{S}_\op{L}^2\op{M}u_\op{L})(\phi_\op{L}\op{M} u_\op{L})
\end{align*}
which factorizes as stated.

Finally, the definition of $\op{R}$ can be rewritten as $\frac{1}{\lambda_{\op{L}}}\op{R}_{\op{L}}=\frac{\op{L}}{\lambda_{\op{L}}}-\op{P}_{\op{L}}$, from which the derivative follows.
\end{proof}

\begin{rema}
The expression of $D^2\lambda_\op{L}$ in Proposition \ref{prop:derivatives} generalizes (and simplifies part of the proof of) Theorem C in \cite{GKLM}. The only part needing additional work from this expression is to work out the splitting of the space into the direct sum of the tangent space to the normalized potentials and the set of coboundaries and constants. The various reformulations of the expression are then classical.
\end{rema}

%%%%%%%%%%%%%%%%%%%%%%%%%%%%%%%%%%%%%%%%%%%%%%%%%%%%%%%%%%%%%%%
%%%%%%%%%%%%%%%%%%%%%%%%%%%%%%%%%%%%%%%%%%%%%%%%%%%%%%%%%%%%%%%
\section{Building estimates from the differentiation formulas}
\label{sec:estimates}

We can now bound all above derivatives in terms of the two fundamental  quantities
\[\tau_\op{L} = \lVert \phi_\op{L}\rVert \lVert u_\op{L}\rVert = \lVert \op{P}_\op{L}\rVert \quad\mbox{and}\quad \gamma_\op{L} = \lVert \op{S}_\op{L}\rVert \]
(recall we normalized $\phi_\op{L}$ to ensure $\lvert \phi_\op{L} u_\op{L}\rvert =1$).

\begin{prop}\label{prop:local-bounds}
At each $\op{L}$ near $\op{L}_0$ we have
\begin{enumerate}
\item $\lVert\pi_\op{L}\rVert \le 1+\tau_\op{L}$ and $\lVert \pi_\op{L}^*\rVert\le 1+\tau_\op{L}$,
\item $\lVert D\lambda_\op{L}\rVert = \tau_\op{L}$,
\item $\lVert D^2\lambda_\op{L}\rVert \le 2\gamma_\op{L}\tau_\op{L}$,
\item $\lVert D\pi_\op{L} \rVert \le 2\gamma_\op{L}\tau_\op{L}$,
\item $\lVert D\op{P}_\op{L} \rVert \le 2\gamma_\op{L}\tau_\op{L}$,
%\item $\Big\lVert D\big(\frac{1}{\lambda}\op{R}\big)_\op{L} \Big\rVert \le \frac{1}{\lVert\lambda_\op{L}\rVert}+\frac{\tau_\op{L} \lVert\op{L}\rVert}{\lVert\lambda_\op{L}\rVert^2}+2\gamma_\op{L}\tau_\op{L}$.
\item\label{enumi:Dgamma} $\lVert D\op{S}_\op{L} \rVert \le \gamma_\op{L}^2(1+3\tau_\op{L})$,
\item $\lVert D^3\lambda_\op{L} \rVert \le 6\gamma_\op{L}^2\tau_\op{L}(1+\tau_\op{L})$,
\item as soon as $\lambda_\op{L}\neq0$, $\Big\lVert D\big[\frac{1}{\lambda} \op{R}\big]_\op{L}\Big\rVert\le \frac{1}{\lvert \lambda_\op{L}\rvert}+\frac{\tau_\op{L}}{\lvert\lambda_\op{L}\rvert^2}\lVert\op{L}\rVert+2\tau_\op{L}\gamma_\op{L}$.
\end{enumerate}
\end{prop}

\begin{proof}
All bounds follow  directly from the expressions given in \ref{prop:derivatives}, for example 
\[\lVert D\lambda_\op{L}(\op{M})\rVert = \lVert \phi_\op{L}\op{M} u_\op{L}\rVert \le \lVert \phi_\op{L}\rVert \lVert u_\op{L}\rVert \lVert\op{M}\rVert = \tau_\op{L} \lVert\op{M}\rVert.\]
To get the equality $\lVert D\lambda_\op{L}\rVert=\tau_\op{L}$, one simply considers perturbations $\op{M}$ of unit norm that send $u_\op{L}$ to vectors of the same norm on which $\phi_\op{L}$ almost realizes its norm (such $\op{M}$ exist by the Hahn-Banach Theorem).
\end{proof}

\begin{rema}
Some of the constants above can be improved if we know more about $\fX$, for example if it is a Hilbert space. Indeed, the two terms given in Proposition \ref{prop:derivatives} for $D\pi_\op{L}(\op{M})(x)$ are its component $-\big[\phi_\op{L}\op{M}\op{S}_\op{L}(x)\big]u_\op{L}$ in the direction of $u_\op{L}$ and $-\phi_\op{L}(x)\big[\op{S}_\op{L}\op{M} u_\op{L}\big]$ in $G_\op{L}$; in a Hilbert space under a good bound on $\tau$,  we know that the two terms are close to being orthogonal and so the norm of their sum must be somewhat lower than the sum of their norm.
This kind of argument can be more generally used in a space with some uniform convexity estimates. We do not pursue these improvements because they only apply in restrictive cases and we expect them to be quite modest.
\end{rema}

We arrive at our core result, which will enable us to control $\tau$ and $\gamma$ and in turn all eigendata.
\begin{coro}\label{coro:core}
We have $\llip \tau\le 2\gamma\tau$ and $\llip \gamma\le \gamma^2(1+3\tau)$.
\end{coro}
Here $\llip$ denotes the metric derivative (also known as the local Lipschitz constant) defined in Section \ref{sec:metric}.
\begin{proof}
By Proposition \ref{prop:local-bounds} we have $\tau_\op{L}=\lVert D\lambda_\op{L}\rVert$, and then $\llip\tau = \llip \lVert D\lambda\rVert \le \lVert D^2\lambda\rVert \le 2\gamma\tau$.
Since $\gamma=\lVert\op{S}\rVert$ we have $\llip \gamma\le \lVert D\op{S}\rVert\le \gamma^2(1+3\tau)$.
\end{proof}

There are \emph{a priori} several ways to combine these bounds together; optimally, one would compare $\tau_\op{L}$ and $\gamma_\op{L}$ to the values $t(r)$, $g(r)$ at $r=\lVert\op{L}-\op{L}_0\rVert$ of the solutions $t,g$ to the differential system
\begin{equation}
\begin{cases}
t'=2tg \\
g' = g^2(1+3t) \\
t(0) = \tau_0 \quad\&\quad g(0)=\gamma_0.
\end{cases}
\label{eq:system}
\end{equation} 
However the solutions of this system are unlikely to have a nice expression, and the explosion time might be difficult to express exactly.

Instead, we accept to loose a little ground for the sake of simplicity and usability. This leads us to the following.
\begin{coro}\label{coro:core2}
We have $\llip (\gamma\tau) \le 6 \gamma^2\tau^2$.
\end{coro}

\begin{proof}
We simply observe
\begin{align*}
\llip (\gamma\tau) &\le (\llip\gamma)\tau +\gamma(\llip \tau) = \gamma^2(1+3\tau)\tau + \gamma\cdot 2\gamma\tau = \gamma^2(3\tau^2+3\tau) \\
  &\le 6\gamma^2\tau^2
\end{align*}
since $1\le \tau$.
\end{proof}

\begin{rema}
Since $t$ is increasing, $g'\ge (1+3\tau_0)g^2$ and the explosion time of \eqref{eq:system} is at most $\frac{1}{(3\tau_0+1)\gamma_0}$ while we will get $\frac{1}{6\tau_0\gamma_0}$. This shows that the loss coming from this relaxation is modest.
\end{rema}

%%%%%%%%%%%%%%%%%%%%%%%%%%%%%%%%%%%%%%%%%%%%%%%%%%%%%%%%%%%%%%%
%%%%%%%%%%%%%%%%%%%%%%%%%%%%%%%%%%%%%%%%%%%%%%%%%%%%%%%%%%%%%%%
\section{End of the proofs of the main results}
\label{sec:proofs}

\begin{proof}[Proof of Theorems \ref{theo:perturbation} and \ref{theo:estimate1}]
First, we want to prove that every $\op{L}\in B(\op{L}_0,r_0)$ has an ASIE for $r_0$ as large as possible. To simplify we will look in one direction at a time: fix some $\op{M}\in\BX$, of norm $1$ say.
Set $\op{L}_r= \op{L}_0+r\op{M}$ and define
\[B=\big\{r_0\in[0,+\infty) \,\big|\, \exists \varepsilon>0,\, \forall r\in [0,r_0], \forall \op{L}\in B(\op{L}_r,\varepsilon): \op{L} \mbox{ has an ASIE } \big\}\]
(recall that implicitly the eigenvalue is required to be analytic, in particular continuous, on this neighborhood of a segment). By Proposition \ref{prop:open}, $B$ is open and $0\in B$.  By definition $B$ is an interval, so $B=[0,r_+)$ for some $r_+\in (0,+\infty]$, which \emph{a priori} depends on $\op{M}$ but that we intend to bound uniformly from below.

By abuse of notation, let $\tau,\gamma:B\to (0,\infty)$ be the functions sending $r$ to $\tau(r):=\tau_{\op{L}_r}$ and $\gamma(r):=\gamma_{\op{L}_r}$ respectively. Since $\lVert\op{M}\rVert=1$,  Corollary \ref{coro:core2} yields again $\llip (\tau\gamma) \le 6\tau^2\gamma^2$ in this notation.

By comparison (see Proposition \ref{prop:comparison}) we thus have $\tau\gamma\le w$ where $w$ is the solution of $w'=6w^2$ with $w(0)=\tau_0\gamma_0$, as long as $w$ is defined. Solving this equation explicitly, we conclude that for all $r\in B$ smaller than $\frac{1}{6\tau_0\gamma_0}$ it holds
\[\tau(r)\gamma(r) \le \Big(\frac{1}{\tau_0\gamma_0}-6r\Big)^{-1}=\frac{\tau_0\gamma_0}{1-6\tau_0\gamma_0 r}.\]

Assume by contradiction that $r_+<\frac{1}{6\tau_0\gamma_0}$. Then $\tau\gamma$ is uniformly bounded from above on $B$. Let us prove that $\tau$ and $\gamma$ are both uniformly bounded from above. First $\tau\ge 1$ so that $\gamma\le \tau\gamma$. Second, denoting by $\gamma_+$ a upper bound for $\gamma$, we have
$\llip \tau\le 2\gamma_+\tau$ so that again by comparison, $\tau(r)\le \tau_0\exp(r\gamma_+)$. 

Proposition \ref{prop:local-bounds} ensures that $\lambda$, $\pi$, $\op{P}$, $\op{S}$ are Lipschitz on a neighborhood of $\{\op{L}_r, r\in B\}$. Since $r_+$ is finite and $\fX$, $\BX$ are complete, these eigendata all have limits when $r\to r_+$. The limit of $\op{P}_{\op{L}_r}$ as $r\to r_+$ must be a rank-one projection to some direction $\langle u_{r_+}\rangle$, where $\op{L}_{r_+}$ has eigenvalue $\lim_{r_+} \lambda_{\op{L}_r}$. The limit of $\pi_{\op{L}_r}$ must be a projection to some subspace $G_{r_+}$ preserved by $\op{L}_{r_+}$. Since the relations $\op{P}_\op{L}\pi_\op{L}=0$ and $\op{P}_\op{L}+\pi_\op{L} =\Id$ pass to the limit, $G_{r_+}$ must be a complement to $\langle u_{r_+}\rangle$. The relation $(\op{L}-\lambda)\op{S}=\pi$ then also passes to the limit, and since $S$ and $\lVert S\rVert$ converge, $\op{L}_{r_+}$ has a simple isolated eigenvalue. Applying Proposition \ref{prop:open} to $\op{L}_{r_+}$ and using continuity of $\lambda$, we see that $r_+\in B$. This is a contradiction since $B$ is open and $r_+:=\sup B$.

At this point, since the bound $r_+\ge 1/6\tau_0\gamma_0$ does not depend on $\op{M}$, we have established that $\lambda$ (and the other eigendata) can be defined and is simple isolated on $B(\op{L}_0,r_0)$ with $r_0=\frac{1}{6\tau_0\gamma_0}$. In addition we get on this ball the bound 
\[\tau_\op{L}\gamma_\op{L} \le \Big(\frac{1}{\tau_0\gamma_0}-6\lVert\op{L}-\op{L}_0\rVert\Big)^{-1}=\frac{\tau_0\gamma_0}{1-6\tau_0\gamma_0 \lVert\op{L}-\op{L}_0\rVert}.\]
When $\lVert\op{L}-\op{L}_0\rVert\le (K-1)/6K\tau_0\gamma_0$, this implies
$\tau_\op{L}\gamma_\op{L} \le K\tau_0\gamma_0$.
In particular, $\tau$ is $2K\tau_0\gamma_0$-Lipschitz, so that
\[\tau\le\tau_0+2K\tau_0\gamma_0 \frac{K-1}{6K\tau_0\gamma_0} \le \tau_0+\frac{K-1}{3}.\]
Then $\lVert D\lambda_\op{L}\rVert\le \tau_\op{L}\le \tau_0+(K-1)/3$, $\lVert D\op{P}_\op{L}\rVert$, $\lVert D\pi_\op{L}\rVert$ and $\lVert D^2\lambda_\op{L}\rVert \le 2\tau_\op{L}\gamma_\op{L}\le 2K\tau_0\gamma_0$, and $\lVert D^3\lambda_\op{L}\rVert \le 6\tau_\op{L}\gamma_\op{L}^2(1+\tau_\op{L}) \le 12\tau_\op{L}^2\gamma_\op{L}^2\le 12K^2\tau_0^2\gamma_0^2$.
\end{proof}

\begin{proof}[Proof of Theorem \ref{theo:spectral-gap}]
The hypothesis ensures at least $\lVert\op{L}-\op{L}_0\rVert\le 1/6\tau_0\gamma_0$, so that $\op{L}$ has an ASIE $\lambda_\op{L}$. Let $x$ be a vector of $G_\op{L}$. Using that $\pi_\op{L}x=x$ and thus $\lVert\pi_0 x\lVert\le (1+\lVert\pi_\op{L}-\pi_0\rVert)\lVert x\rVert$ we get
\begin{align*}
\lVert \op{L}x\rVert &= \lVert \op{L}_0\pi_0x+\op{L}_0(\pi_\op{L}x-\pi_0x)+(\op{L}-\op{L_0})x\rVert \\
  &\le \lvert\lambda_0\rvert (1-\delta_0) \lVert\pi_0x\rVert+\lVert\op{L}_0\rVert \lVert\pi_\op{L}-\pi_0\rVert\lVert x\rVert+\lVert\op{L}-\op{L}_0\rVert \lVert x\rVert \\
  &\le  \Big(\lvert\lambda_0\rvert (1-\delta_0) + \lVert\pi_0-\pi_\op{L}\rVert\big(\lvert\lambda_0\rvert(1-\delta_0) + \lVert\op{L}_0\rVert\big) + \lVert\op{L}-\op{L}_0\rVert\Big) \lVert x\rVert.
\end{align*}
Now, if $\lVert\op{L}-\op{L}_0\rVert\le (K-1)/6K\tau_0\gamma_0$ for some $K>1$, we have $\lVert\pi_\op{L}-\pi_0\rVert \le \frac{1}{3}(K-1)$ and, using the explicit-remainder first order Taylor formula of Theorem \ref{theo:estimate1} and $\lVert D\lambda_0\rVert=\tau_0$ we get
\[\lvert\lambda_\op{L}\rvert  \ge \lvert\lambda_0\rvert-\frac{(K-1)}{6K\gamma_0}-\frac{(K-1)^2}{36K\tau_0\gamma_0}
\]
so that to ensure $\lVert\op{L}x\rVert \le (1-\delta)\lVert\lambda_\op{L}\rVert \lVert x\rVert$ it suffices to have
\[\lvert\lambda_0\rvert (1-\delta_0) + \frac13(K-1)\big(\lvert\lambda_0\rvert(1-\delta_0) + \lVert\op{L}_0\rVert\big) + \frac{K-1}{6K\tau_0\gamma_0} \le (1-\delta)\Big(\lvert\lambda_0\rvert-\frac{(K-1)}{6K\gamma_0}-\frac{(K-1)^2}{36K\tau_0\gamma_0}\Big)\]
or equivalently,
\[6K \lvert\lambda_0\rvert(\delta-\delta_0)+a K(K-1)+\frac{K-1}{\tau_0\gamma_0}+\frac{(K-1)(1-\delta)}{\gamma_0}+\frac{(K-1)^2(1-\delta)}{6\tau_0\gamma_0} \le0 \]
which, writing $K=K-1+1$, can be rewritten as:
\[(K-1)^2\big(a+\frac{1-\delta}{6\tau_0\gamma_0} \big) +(K-1)\big(6\lvert\lambda_0\rvert(\delta-\delta_0) + a+\frac{1-\delta}{\gamma_0}+\frac{1}{\tau_0\gamma_0} \big)+6\lvert\lambda_0\rvert (\delta-\delta_0)\le 0.\]
\end{proof}

\begin{proof}[Proof of Corollary \ref{coro:gap-short}]
We observe that increasing the first two coefficients in the polynomial of Theorem \ref{theo:spectral-gap} must reduce the value of its positive root. We thus seek simple upper bounds for these two first coefficients. 

We also observe that in all the above we can replace $\gamma_0$ by any larger number $\gamma'_0$, as soon as we make the replacement in both the hypotheses and the conclusions. Here we use $\gamma_0\le \gamma_0'=\lVert\pi_0\rVert/\delta_0$ obtained by 
\[\lVert S_0\rVert \le \lVert(1-\op{L}_0)^{-1} \rVert \lVert\pi_0\rVert\le \lVert\pi_0\rVert\sum_{k\ge0} \lVert (\op{L}_0)_{|G_0}^{k} \rVert \le \lVert\pi_0\rVert\sum_{k\ge0} (1-\delta_0)^k = \lVert\pi_0\rVert/\delta_0.\]
Then we have
\[a+\frac{1-\delta}{6\tau_0\gamma'_0} \le 4+\frac16 = \frac{25}{6}.\]
Then discarding the negative term $\delta-\delta_0$, we have
\[6\lvert\lambda_0\rvert(\delta-\delta_0)+a+\frac{1-\delta}{\tau_0}+\frac{1}{\tau_0\gamma'_0} \le 6.\]
It follows that under the extra assumptions of Corollary \ref{coro:gap-short} we can replace $\rho(\delta)$ in Theorem \ref{theo:spectral-gap} by the root of
\[\frac{25}{6} X^2 + 6 X + 6(\delta-\delta_0),\]
which (factoring $6$ and using $\sqrt{1+x}\le 1+x/2$) satisfies
\[\rho'(\delta)=\frac{-6+\sqrt{36+100(\delta_0-\delta)}}{\frac{25}{3}}\le \delta_0-\delta.\]
\end{proof}

\begin{proof}[Proof of Corollary \ref{coro:gap-ne}]
By hypothesis $\op{L}_0$ has a spectral gap of some size $\delta_0$ with constant $C_0$, and it follows that some power $n_0=O_{*C_0,\delta_0^{-1}}(1)$ of $\op{L}_0$ has a spectral gap (of arbitrary size, say  $1/2$) with constant $1$.
We have $\tau_{\op{L}_0^{n_0}}=\tau_0$. Writing \[\lambda_0^{n_0}-\op{L}_0^{n_0} = (\lambda_0-\op{L}_0)\sum\lambda_0^k\op{L}_0^{n_0-1-k}\]
and observing that $\lVert\op{L}_0\rVert$ is controlled by $\lambda_0$, $\tau_0$ and $C_0$, we get $\gamma_{\op{L}_0^{n_0}}= O_{*C_0,\delta_0^{-1},\lvert\lambda_0\rvert,\tau_0}(\gamma_0)$.
Applying Corollary \ref{coro:gap-short} for all $\op{M}\in\BX$ such that 
\[\lVert(\op{L}_0+\op{M})^{n_0}-\op{L}_0^{n_0}\rVert = O_{*C_0,\delta_0^{-1},\tau_0,\lvert\lambda_0\rvert}(1)\]
we have that $(\op{L}_0+\op{M})^{n_0}$ has a spectral gap (of size $1/4$ say) with constant $1$. This implies that $\op{L}_0+\op{M}$ has a spectral gap of size $O_{*n_0}(1)=O_{*C_0,\delta_0^{-1}}(1)$, with a constant $O_{*C_0,\delta_0^{-1},\tau_0,\lvert\lambda_0\rvert}(1)$.

Developing $(\op{L}_0+\op{M})^{n_0}$, since $\lVert\op{L}_0\rVert=O_{*C_0,\delta_0^{-1},\tau_0,\lvert\lambda_0\rvert}(1)$, we see that
\[(\op{L}_0+\op{M})^{n_0}-\op{L}_0^{n_0} =O_{*C_0,\delta_0^{-1},\tau_0,\lvert\lambda_0\rvert}(\lVert\op{M}\rVert)\]
and the spectral gap is ensured for $\op{M}=O_{*C_0,\delta_0^{-1},\tau_0,\lvert\lambda_0\rvert}(1)$.
\end{proof}

\bibliographystyle{amsalpha}
\bibliography{perturbation}

\end{document}